\documentclass{amsart}
\usepackage{all2021}
\renewcommand{\epsilon}{\varepsilon}

\begin{document}
\title{Countably many asymptotic tensor ranks}

\author{Andreas Blatter}
\address{Mathematical Institute, University of Bern,
Alpeneggstrasse 22, 3012 Bern, Switzerland}
\email{andreas.blatter@unibe.ch}

\author{Jan Draisma}
\address{Mathematical Institute, University of Bern, Sidlerstrasse 5, 3012
Bern, Switzerland; and Department of Mathematics and Computer Science,
P.O. Box 513, 5600 MB, Eindhoven, the Netherlands}
\email{jan.draisma@unibe.ch}

\author{Filip Rupniewski}
\address{Mathematical Institute, University of Bern,
Alpeneggstrasse 22, 3012 Bern, Switzerland}
\email{filip.rupniewski@unibe.ch}

\thanks{AB and FR are supported by Swiss National Science Foundation
(SNSF) project grant 200021\_191981, and JD is partially supported by
that grant and partially supported by Vici grant 639.033.514 from the
Netherlands Organisation for Scientific Research (NWO). JD thanks the
Institute for Advanced Study for the excellent working conditions,
under which part of this project was carried out.}

\maketitle

\begin{abstract}
In connection with recent work on gaps in the asymptotic subranks of
complex tensors the question arose whether the number of nonnegative real
numbers that arise as the asymptotic subrank of some complex tensor is
countable. In this short note we settle this question in the
affirmative, for all tensor invariants that are algebraic in
the sense that they are invariant under field automorphisms
of the complex numbers. 
\end{abstract}

\section{Introduction and Result}

\subsection{Tensor invariants and their asymptotic
counterparts}

Let $d \in \ZZ_{\geq 0}$. An {\em invariant} of $d$-way tensors is
the data of a function
\[ f=f_{V_1,\ldots,V_d}: V_1 \otimes \cdots \otimes V_d \to \RR_{\geq 0} \]
for every choice of finite-dimensional complex vector spaces
$V_1,\ldots,V_d$, satisfying the condition that whenever $\phi_i:V_i
\to W_i$ are linear isomorphisms for $i=1,\ldots,d$, we have
\[ f_{V_1,\ldots,V_d}=f_{W_1,\ldots,W_d} \circ \phi_1
\otimes \cdots \otimes \phi_d. \]
We call $f$ {\em algebraic} if, for every choice of $n_1,\ldots,n_d \in
\ZZ_{\geq 0}$, any $T \in \CC^{n_1} \otimes \cdots \otimes \CC^{n_d}$ and
any field automorphism $\sigma$ of $\CC$, we have $f(\sigma(T))=f(T)$,
where $\sigma(T)$ is the tensor obtained by applying $\sigma$ to all
entries of $T$.

The {\em asymptotic counterpart} $\widetilde{f}$ of an invariant of $d$-way
tensors is another invariant of $d$-way tensors defined by
\[ \widetilde{f}(T):=\lim_{n \to \infty} \sqrt[n]{f(T^{\boxtimes n})}, \]
provided that this limit exists. Here $T^{\boxtimes n}$ is the $n$-th
{\em vertical tensor product} of $T \in V_1 \otimes \cdots \otimes
V_d$, obtained by taking the $n$-th tensor power of $T$ and regarding
this as a $d$-way tensor in 
\[ (V_1^{\otimes n}) \otimes \cdots \otimes (V_d^{\otimes n}). \] 
A sufficient condition for $\widetilde{f}$ to be defined is that
$f$ is sub-multiplicative, i.e., $f(T \boxtimes S) \leq f(T) \cdot
f(S)$ for all $S$ and $T$ (this follows from Fekete's Lemma
\cite{Fekete1923}).

\subsection{The result}

We will establish the following fundamental result, which, as we will
see in Section~\ref{sec:Ranks}, in particular applies to asymptotic
versions of all known versions of rank.

\begin{thm}
Let $f$ be an algebraic invariant of $d$-way tensors, and assume that
its asymptotic counterpart $\widetilde{f}$ exists. Then
\begin{enumerate}
\item also $\widetilde{f}$ is an algebraic invariant of
$d$-way tensors; and 
\item the set of values of $f$, as $V_1,\ldots,V_d$ run through all complex
vector spaces and $T$ runs through $V_1 \otimes \cdots \otimes V_d$,
is at most countably infinite.
\end{enumerate}
\end{thm}

\begin{cor} \label{cor:Countable}
The set of asymptotic ranks, the set of asymptotic subranks,
and the set of asymptotic geometric ranks of complex
$d$-tensors is countable.
\end{cor}

\begin{proof}[Proof of the corollary]
As we will see in Section~\ref{sec:Ranks}, tensor rank, subrank, and
geometric rank are all algebraic, and hence, by the first item of the
theorem, so are their asymptotic counterparts $\widetilde{f}$. By the
second item applied to $\widetilde{f}$, it takes on countably many
values. 
\end{proof}

\subsection{Relations to recent literature}

In \cite{Christandl22}, it is proved that the first few values of the
asymptotic subrank of a $d$-tensor (regardless of the field) are $0,1,$
and $2^{h(1/d)}$, where $h:(0,1) \to \RR$ is the binary entropy function
defined by
\[ h(p):=-p\log_2(p) - (1 - p) \log_2 (1 - p), \]
and similar gaps for various other notions of asymptotic rank
are found. 

This led to the question, posed at the {\em Workshop on geometry
and complexity theory} in Toulouse, April 2022, of what the set of
values of various asymptotic rank functions looks like.

One of the motivation to study asymptotic rank are questions and conjectures related to the exponent of matrix multiplication tensor $M_{\langle n \rangle}$.
Asymptotic rank of an order 3 tensor $T$ may be interpreted as a generalization of the limit with respect to $n$ of the rank of matrix multiplication tensor $M_{\langle n \rangle} \in \CC^{n^2} \otimes \CC^{n^2} \otimes \CC^{n^2}$.
See Strassen’s Asymptotic Rank Conjecture, \cite{Strassen94} and
\cite[Questions 1.5, 1.6]{Conner21} which implicitly already
appeared in  \cite{Strassen88}, \cite{Strassen91}, \cite{Strassen94} and
later in \cite{Buergisser97}, \cite{Christandl23}.

Corollary
\ref{cor:Countable} says that, for tensors
over $\CC$, the set of values of various asymptotic rank functions  is at most countable. It does not, however,  explain the existence of gaps as
found in \cite{Christandl22}.

For tensors over a fixed finite field, it is evident that any tensor
invariant of $d$-tensors takes at most countably many values. In our
recent paper \cite{Draisma22c}, we prove the much deeper statement that
any restriction-monotone real-valued function on $d$-tensors {\em over a
fixed finite field} takes a {\em well-ordered} set of values. This does
explain the existence of gaps, though it does not exclude the possibility
that a set like $\{3-1/n \mid n \in \ZZ_{\geq 1}\}$ might be the set of
values of such a function.

Over $\CC$ (or any other field), restriction-monotone functions that,
in addition, are lower semi-continuous in the Zariski topology, also
take a well-ordered set of values, see \cite[Remark 1.3.4]{Draisma22c}.

\section{Algebraicity of tensor ranks} \label{sec:Ranks}

\subsection{Three well-known notions of rank}

Let $V_1,\ldots,V_d$ be finite-dimensional complex vector spaces. 

\begin{de}
The {\em rank} of $T \in V_1 \otimes \cdots \otimes V_d$ is the smallest
$r$ such that $T$ can be written as
\[ T=\sum_{i=1}^r v_{i,1} \otimes \cdots \otimes v_{i,d}. \]
Dually, the {\em subrank} of $T$ is the largest $r$ such
that there exist linear maps $\phi_i:V_i \to \CC^r$ with 
\[ (\phi_1 \otimes \cdots \otimes \phi_d)(T)=\sum_{i=1}^r
e_i \otimes \cdots \otimes e_i. \]
Finally, the {\em geometric rank} of $T$ \cite{Kopparty20} equals the
codimension in $V_1^* \times \cdots \times V_{d-1}^*$ of the
algebraic variety 
\[ X(T)=\{(x_1,\ldots,x_{d-1}) \in V_1^* \times \cdots
\times V_{d-1}^* \mid
T(x_1,\ldots,x_{d-1},\cdot)\equiv 0\}. \qedhere \]
\end{de}

\subsection{Algebraicity}

\begin{lm}
Rank, subrank, and geometric rank are algebraic tensor
invariants. 
\end{lm}

\begin{proof}
For tensor rank, we note that applying a field automorphism $\sigma$ to all
vectors in a decomposition of $T$ as a sum of $r$ tensors product of
vectors yields such a decomposition of $\sigma T$. For
subrank, we note that if 
\[ (A_1 \otimes \cdots \otimes A_d)T = \sum_{i=1}^r e_i^{\otimes
d}, \]
then 
\[ (\sigma A_1) \otimes \cdots \otimes (\sigma A_d))(\sigma
T)=\sum_{i=1}^r e_i^{\otimes d}. \]
For geometric rank, we first observe that the action of $\Aut(\CC)$
on $\CC$ and on each $\CC^{n_i}$ naturally yields an action on the dual
space $(\CC^{n_i})^*$ given by $(\sigma x)(v):=\sigma(x(\sigma^{-1}
v))$. With respect to this action we have 
\begin{align*} 
\sigma X(T)&=\{(\sigma x_1,\ldots,\sigma x_{d-1}) \mid
\forall y \in V_d^*:
T(x_1,\ldots,x_{d-1},y) =0\}\\
&=\{(x_1,\ldots,x_{d-1}) \mid
\forall y \in V_d^*:
T(\sigma^{-1} x_1,\ldots, \sigma^{-1} x_{d-1}, \sigma^{-1} y) = 0\}\\
&=\{(x_1,\ldots,x_{d-1}) \mid \forall y \in V_d^*:
\sigma T(\sigma^{-1} x_1,\ldots, \sigma^{-1} x_{d-1},
\sigma^{-1} y) = 0\}\\
&=X(\sigma T). 
\end{align*}
In the last step, we used that the natural action of $\Aut(\CC)$ on
tensors agrees with the action of $\Aut(\CC)$ on multilinear forms given
by $(\sigma,T) \mapsto \sigma \circ T \circ (\sigma^{-1})^d$.

Since field automorphisms are continuous in the Zariski
topology and dimension is a topological invariant, this
implies that the geometric rank of $T$ equals that of
$\sigma T$. 
\end{proof}

We leave it to the reader to verify that many other notions of tensor
rank, such as slice rank and partition rank, are also algebraic.

\section{Proof of the theorem}

\begin{proof}
We first show that if $f$ is an algebraic tensor invariant, then so
is $\widetilde{f}$. To this end, let $T \in \CC^{n_1} \otimes \cdots
\otimes \CC^{n_d}$ and let $\sigma$ be a field automorphism of $\CC$. Set
$r:=\tilde{f}(T)$ and let $\epsilon>0$. Then there exists an
$n_0$ such that for $n > n_0$ we have 
\[ \left|\sqrt[n]{f(T^{\boxtimes n})}-r\right|< \epsilon. \]
We then have 
\[ \left|\sqrt[n]{f((\sigma T)^{\boxtimes n})}-r\right|
= 
\left|\sqrt[n]{f(\sigma (T^{\boxtimes n}))}-r\right|
=
\left|\sqrt[n]{f(T^{\boxtimes n})}-r\right|
<\epsilon, \]
where the first step uses that the entries of $T^{\boxtimes n}$ are
polynomial functions, defined over $\ZZ$, of the entries of $T$, and
that these functions commute with the field automorphism $\sigma$;
and the last step uses that $f$ is algebraic.

To see that the number of values of an algebraic tensor invariant is
at most countable, we proceed as follows.  As the tensor invariant $f$
is invariant under linear isomorphisms, any $d$-way tensor has the
same $f$-value as some tensor in a tensor product $\CC^{n_1} \otimes
\cdots \otimes \CC^{n_d}$. Since the number of tuples $(n_1,\ldots,n_d)
\in \ZZ_{\geq 0}$ is countable, it suffices to show that for fixed
$n_1,\ldots,n_d$, $f$ takes at most a countable number of values on
$\CC^{n_1} \otimes \cdots \otimes \CC^{n_d}$.

To this end, we claim that the value of $f$ on a $d$-way tensor
$T=(t_{i_1,\ldots,i_d})_{i_1,\ldots,i_d}$ depends only on the ideal
\[ I(T):=\{f \in R:=\QQ[(x_{i_1,\ldots,i_d})_{i_1,\ldots,i_d}] \mid
f(T)=0\} \]
of algebraic relations over $\QQ$ among the entries of $T$.
Indeed, let $T'$ be another tensor with $I(T)=I(T')$, and let $L,L'$
be the subfields of $\CC$ generated by the entries of $T$ and of $T'$,
respectively. Since $L$ is isomorphic to the fraction field of the quotient
$R/I(T)$, and similarly for $L'$, there is a unique field isomorphism
$\sigma: L \to L'$ that maps each entry $t_{i_1,\ldots,i_d}$ of $T$ to the
corresponding entry $t'_{i_1,\ldots,i_d}$ of $T'$. 

Any field isomorphism between finitely generated subfields of $\CC$
extends to an automorphism of $\CC$. (Indeed, the transcendence degree of
$\CC$ over $L$ and over $L'$ are equal (infinite) cardinal
numbers (namely, the cardinality of $\RR$), so if
we choose transcendence bases $B,B'$ of $\CC$ over $L,L'$, respectively,
then there exists a bijection $B \to B'$. Extend $\sigma:L \to L'$
to the unique isomorphism $L(B) \to L'(B')$ that agrees with this
bijection on $B$.  And then apply \cite[Theorem 6]{Yale66} to see that
this isomorphism extends to an isomorphism of the algebraic closure $\CC$
of $L(B)$ into the algebraic closure of $L(B')$, which is also $\CC$.)

Denote by $\sigma$ an extension of $\sigma$ to an
automorphism of $\CC$. Then $\sigma T=T'$ by construction,
and therefore $f(T)=f(\sigma T)=f(T')$ since $f$ is
algebraic.

Finally, $I(T)$ is an ideal in a polynomial ring over $\QQ$ with
a finite number of variables, and hence, by Hilbert's basis theorem,
generated by a finite number of polynomials. Since the number
of polynomials in said polynomial ring is countable, so is
the number of finite subsets of that polynomial ring, and
so is the number of ideals in it. We conclude that $I(T)$ takes only
countably many values, hence, by the above, so does $f$.
\end{proof}

\bibliographystyle{alpha}
\bibliography{diffeq,draismapreprint}

\end{document}